\documentclass[12pt]{amsart}

\usepackage[a4paper,margin=1.2in]{geometry}

\RequirePackage{amsthm,amsmath,amsfonts,amssymb}
\RequirePackage[authoryear]{natbib}
\RequirePackage[colorlinks,citecolor=blue,urlcolor=blue]{hyperref}
\RequirePackage{graphicx}

\usepackage{geometry}
\usepackage{enumerate}
\usepackage{mathtools}
\usepackage{amsfonts}
\usepackage{mathgreeks}
\usepackage{tikz}
\usetikzlibrary{patterns}
\usepackage{diagbox}

\usepackage{mathrsfs}
\usepackage{bbm}

\newtheorem{theorem}{Theorem}

\newtheorem{lemma}[theorem]{Lemma}
\newtheorem{proposition}[theorem]{Proposition}

\newtheorem{example}{Example}
\newtheorem{claim*}{Claim}

\numberwithin{equation}{section}
\numberwithin{example}{section}
\numberwithin{table}{section}

\numberwithin{figure}{section}
\renewcommand{\theenumi}{\roman{enumi}}

\newcommand{\figref}[1]{Figure~\ref{fig:#1}}

\newcommand{\secref}[1]{Section~\ref{sec:#1}}

\newcommand{\appref}[1]{Appendix~\ref{app:#1}}
\newcommand{\appsref}[1]{Appendices~\ref{app:#1}}

\newcommand{\exref}[1]{Example~\ref{ex:#1}}

\newcommand{\lemref}[1]{Lemma~\ref{lem:#1}}

\newcommand{\propref}[1]{Proposition~\ref{prop:#1}}

\newcommand{\thmref}[1]{Theorem~\ref{thm:#1}}

\newcommand{\tabref}[1]{Table~\ref{tab:#1}}

\newcommand{\cH}{\mathcal{H}}

\newcommand{\cR}{\mathcal{R}}

\newcommand{\cB}{\mathcal{B}}
\newcommand{\cP}{\Omega}

\newcommand{\bd}{d}

\newcommand{\Xdata}{X}

\newcommand{\cF}{\mathcal{F}}

\newcommand{\cJ}{\mathcal{J}}
\newcommand{\cS}{\mathcal{S}}
\newcommand{\cU}{\mathcal{U}}
\newcommand{\cV}{\mathcal{V}}
\newcommand{\cW}{\mathcal{W}}
\newcommand{\cI}{\mathcal{I}}

\renewcommand{\P}{\text{P}}
\newcommand{\cA}{\mathcal{A}}

\newcommand{\cE}{\mathcal{E}}

\allowdisplaybreaks

\title{Interpolation and coherence in multiple testing}

\author{Ninh Tran}
\address{
School of Mathematics and Statistics,
The University of Melbourne, Australia
}
\email{ninht@student.unimelb.edu.au}

\author{Dennis Leung}
\address{
School of Mathematics and Statistics,
The University of Melbourne, Australia
}
\email{dennis.leung@unimelb.edu.au}

\author{Jesse Hemerik}
\address{
Econometric Institute, Erasmus University Rotterdam, the Netherlands
}
\email{hemerik@ese.eur.nl}

\begin{document}

\maketitle

\begin{abstract}
In the literature on  multiple testing methods that control the familywise error rate, a fundamental concept is ``coherence''. A procedure is coherent if it is ``logically consistent'' (Hochberg and Tamhane, 1987) in the sense that if a hypothesis is  rejected, then all hypotheses that imply it are also rejected. If a method is not coherent, then it can be made coherent without increasing the familywise error rate.
Recently, the concept of coherence was generalized to a much broader class of multiple testing procedures, namely all methods that provide simultaneous confidence bounds for false discovery proportions. 
Such procedures were defined to be coherent if they cannot be further improved through a process called ``interpolation''.
However, we show that such a coherent procedure can sometimes be further improved by exploiting remaining logical inconsistencies. This leads to what we call a ``perfectly coherent'' method.
Thus, interpolation does not necessarily lead to perfect coherence. However, we provide conditions under which it does. Fortunately, many methods based on interpolation turn out to be perfectly coherent.
\end{abstract}

\section{Introduction}

When multiple hypotheses are tested, familywise error rate (FWER) control means ensuring that with probability at least $1-\alpha$, no type I errors are committed.
In the literature on FWER control, \emph{coherence}  is a prominent concept \citep{gabriel1969simultaneous, Sonnemann1982, hochberg1987multiple, Grechanovsky1999,Romano2011, lehmann2022testing}. 
A FWER method is called coherent if the following holds: if a hypothesis $\cH$ is rejected, then all considered hypotheses that imply it are also rejected. This is typically considered to be a desirable property, since if $\cH$ is false, then necessarily, all hypotheses that imply it are also false.

More recently, a popular approach to multiple testing has been to provide confidence statements about false discovery proportions. One approach is false discovery exceedance (FDX) control, which means ensuring that with probability at least $1-\alpha$, the proportion of false positives (FDP) does not exceed some fixed small constant. Another framework, which in fact generalizes FWER and FDX control, is to provide simultaneous \emph{true discovery guarantees} (TDG)---also known as simultaneous FDP control. If $\cI$ is the set of indices of all hypotheses and $\cI_1$ is the set of indices of false hypotheses, then a TDG procedure provides a lower bound $d(\cS)$ for the number $|\cS \cap \cI_1|$ for each set $\cS$. These bounds are simultaneous if with probability at least $1-\alpha$, for all $\cS \subseteq \cI$, $d(\cS)\leq |\cS\cap \cI_1|$.

In the words of \citet[][p.44]{hochberg1987multiple}, coherence is a ``logical consistency'' property: if a FWER  procedure is coherent, then purely based on the rejections it has made, no further rejections are implied. 
As pointed out in \citet{Goeman2021}, it is natural to extend the concept of coherence to TDG procedures. For example, if $d(\{1,2,3\})=2$, then we have inferred that at least two hypotheses among $\cH_1$, $\cH_2$ and $\cH_3$ are false. Hence, we  would want $d(\{2,3\})$ and $d(\{2,3,4\})$ to be at least 1, and $d(\{1,2,3,4\})$ to be at least 2 for instance. This type of reasoning is an example of \emph{interpolation} \citep{blanchard2020post}, a technique which is also used in e.g. \citet{van2004augmentation}, \citet{farcomeni2009generalized}, \citet{blain2022notip} and \citet{andreella2023permutation}. Using the concept of interpolation, \citet{Goeman2021} gave the following definition of coherence of TDG procedures: a TDG method is coherent if it cannot be improved by interpolation.

The main message of the present paper, however,  is that if a method satisfies the mentioned definition of coherence, it may still contain logical inconsistencies. These may be exploited to further improve the bounds, leading to what we call a \emph{perfectly coherent} method (\secref{perfect_coherence}), which is closely related to the notion of optimality considered by \citet{blanchard2020post} in the setting of \textit{reference families}. A perfectly coherent method is ``logically consistent'' in the spirit of \citet[][p.44]{hochberg1987multiple}. 
 Our result has consequences for the validity of Corollary 1 of \citet{Goeman2021}, which says that every coherent procedure is a closed testing procedure. That corollary is only valid if we replace ``coherent'' by ``perfectly coherent'' there. All other results in that paper remain correct. We also provide a condition on the structure of the TDG procedure under which interpolation leads to a perfectly coherent method (\secref{TU_equivalence}), extending beyond the settings previously studied by \citet{Durand2020}. Finally, in the Discussion (\secref{discussion}), we show how making methods perfectly coherent can be valuable in meta-analysis, where coherence and perfect coherence may not coincide. 

\section{Interpolation and coherence}\label{sec:coherence}
Let $\Xdata$ denote data generated from an unknown probability distribution $\P\in \cP$. 
For a finite index set $\cI\subseteq\mathbb N$, suppose we wish to test whether $\P\in\cH_i$ for each $i\in\cI$, where $\cH_i\subseteq\cP$ is a hypothesis. Accordingly, let
\begin{equation*}
    \cI_0
    = \{i \in \cI : \P \in \cH_{i}\}
    \quad
    \text{and}
    \quad
    \cI_1
    = \cI \backslash  \cI_0
\end{equation*}
be the index sets of true and false hypotheses, respectively. Let $\bd:2^{\cI} \rightarrow \mathbb{Z}$ 
denote a random set function depending on $\Xdata$. For a fixed significance level $\alpha\in(0,1)$, we call $\bd$ a true discovery guarantee (TDG) procedure if
\begin{equation*}
    \P\left(
        \bd(\cS)
        \leq
        |\cS\cap\cI_1|
        \text{ for all }\cS\subseteq\cI
    \right)
    \geq 1-\alpha \quad \text{for all $\P\in\cP$. }
\end{equation*}
Without loss of generality, we restrict our attention throughout to TDG procedures satisfying $0\leq \bd(\cS)\leq |\cS|$ for all $\cS\subseteq\cI$ and $\bd(\varnothing)=0$, since $0\leq |\cS\cap\cI_1|\leq |\cS|$ and $|\varnothing\cap\cI_1|=0$.

With probability at least $1-\alpha$, the TDG procedure $\bd$ provides a lower bound on the number of true discoveries for all subsets of $\cI$ simultaneously, i.e.,
\begin{equation}\label{constraints}
    \bd(\cS) \leq |\cI_1 \cap \cS|
    \quad \text{for all } \cS \subseteq \cI.
\end{equation}
Through a process called \emph{interpolation} \citep{blanchard2020post,Goeman2021},
$\bd$ can be uniformly improved by exploiting relationships among the constraints \eqref{constraints} imposed by its bounds, while preserving its $(1-\alpha)$ guarantee. The TDG procedure resulting from interpolation, denoted by $\Bar{\bd}$, is defined by
\begin{equation*}
    \Bar{\bd}(\cS) = \max_{\cU \subseteq \cI} \{ \bd(\cU) - |\cU \backslash \cS| + \bd(\cS \backslash \cU) \}
    \quad \text{for all } \cS \subseteq \cI.
\end{equation*}
Since $\cU=\cS$ is one of the values over which the maximum is taken, it follows that
\begin{equation*}
    \Bar{\bd}(\cS)\geq \bd(\cS)
    \quad\text{for all } \cS\subseteq \cI.
\end{equation*}
Intuitively, interpolation considers every $\cU \subseteq \cI$ and takes the largest lower bound on $|\cS \cap \cI_1|$ implied by combining the constraints $\bd(\cU) \leq |\cI_1 \cap \cU|$ and $\bd(\cS \backslash \cU) \leq |\cI_1 \cap (\cS \backslash \cU)|$. The subtraction of $|\cU \backslash \cS|$ accounts for true discoveries in $\cU$ that may lie outside $\cS$.

\citet{Goeman2021} defined $\bd$ to be \emph{coherent} if it cannot be uniformly improved by interpolation; that is,
\begin{equation*}
    \bd(\cS) = \Bar{\bd}(\cS) \quad \text{for all $\cS \subseteq \cI$.}
\end{equation*} 
Interpolation is not necessarily a one-off process, since an interpolated procedure may be interpolated again. Starting from an arbitrary $\bd$, interpolation can therefore be applied repeatedly until no further improvement is possible. Let $\check{\bd}$ denote the resulting procedure. By design, $\check{\bd}$ is coherent and satisfies
\begin{equation*}
    \check{\bd}(\cS)\geq \Bar{\bd}(\cS) \quad \text{for all $\cS\subseteq\cI$,}
\end{equation*}
since $\Bar{\bd}$ is obtained after only one round of interpolation. The procedure $\check{\bd}$ is well defined because repeated interpolation must eventually terminate at a coherent procedure. Indeed, since $\cI$ is finite, each integer-valued bound for $\cS\subseteq\cI$ can increase only finitely many times, as interpolation cannot produce a bound exceeding $|\cS|$.

\section{Perfect coherence}\label{sec:perfect_coherence}
The coherent procedure $\check{\bd}$ is obtained by repeatedly interpolating $\bd$ until no further improvement is possible. A natural question is whether $\check{\bd}$ provides the largest possible bounds implied by the constraints \eqref{constraints} imposed by $\bd$. We show that this is not necessarily the case, and a key takeaway is that all such improvements are exhausted by a \textit{maximization} operation closely related to that considered by \citet{blanchard2020post}. The proofs of the theoretical results below are given in \appref{perfect_coherence}.

The collection of all possible candidates for $\cI_1$ consistent with \eqref{constraints} is defined by
\begin{equation*}
    \cA
    = \cA(\bd) =
    \left\{
    \cJ \subseteq \cI :
    \bd(\cS) \leq |\cJ \cap \cS|
    \text{ for all } \cS \subseteq \cI
    \right\}.
\end{equation*}
Note that $\cI \in \cA$, so $\cA \neq \varnothing$. Define the maximization of $\bd$ to be the random set function $\tilde{\bd}:2^{\cI} \to \mathbb{Z}$ given by
\begin{equation*}
    \tilde{\bd}(\cS) = \min_{\cJ \in \cA} |\cJ \cap \cS| \quad \text{for all $\cS \subseteq \cI$.}
\end{equation*}
\begin{lemma}\label{lem:d_hat_improve}
   The procedure $\tilde{\bd}$ has a $(1-\alpha)$ true discovery guarantee and satisfies $\tilde{\bd}(\cS) \geq \check{\bd}(\cS)$ for all $\cS \subseteq \cI$.
\end{lemma}
\lemref{d_hat_improve} states that maximization yields a TDG procedure $\tilde{\bd}$ that dominates the coherent procedure $\check{\bd}$ obtained through repeated interpolation. Intuitively, this is because maximization exhausts all logical improvements by setting $\tilde{\bd}(\cS)$ to be the largest lower bound on $|\cI_1 \cap \cS|$ implied by \eqref{constraints}. Maximization is therefore a one-off operation, since applying it again cannot yield any further improvement.
\begin{proposition}\label{prop:max_is_max}
    Applying maximization to a TDG procedure that has already been maximized leaves it unchanged.
\end{proposition}
We call $\bd$ a \emph{perfectly coherent} TDG procedure if maximization cannot improve it further. That is,
\begin{equation*}
\bd(\cS) = \tilde{\bd}(\cS) \quad \text{for all $\cS \subseteq \cI$.}
\end{equation*}
By \propref{max_is_max}, every TDG procedure can be made perfectly coherent by a single application of maximization. Moreover, by \lemref{d_hat_improve}, interpolation cannot improve a perfectly coherent TDG procedure. Hence, every perfectly coherent TDG procedure is also coherent.

Conceptually, perfect coherence is more faithful to the notion of coherence for FWER procedures than to the coherence introduced by \citet{Goeman2021}. In the FWER setting, a procedure is said to be coherent, or ``logically consistent'' \citep[p.44]{hochberg1987multiple}, if no further rejections are implied by those already made. The natural analogue for TDG procedures is that no further improvements to the true discovery bounds are implied by the bounds already obtained, which is the notion of perfect coherence.

In \exref{dI}, we present a TDG procedure $\bd$ for which the corresponding perfectly coherent procedure $\tilde{\bd}$ provides a uniform improvement over the coherent procedure $\check{\bd}$. In \appref{PC}, we show that \exref{dI} is not merely pathological, since the construction of the procedure $\bd$ appearing in the example can be motivated by a partial conjunction test \citep{Benjamini2008}.

\begin{example}\label{ex:dI}
Let $\cI= \{1,2,3\}$ and let $\cI_0 \subseteq \cI$ be arbitrary. Suppose the data $\Xdata$ has distribution $\P$ where
\begin{equation*}
    \P(\Xdata=0)=1-\alpha \quad \text{and} \quad \P(\Xdata=1)=\alpha.
\end{equation*}
Define the TDG procedure
\begin{equation*}
    \bd(\cS)=
    \begin{cases}
        \Xdata & \text{if } |\cS|=2\\
        0 & \text{otherwise}
    \end{cases}
    \quad \text{for all $\cS \subseteq \cI$.}
\end{equation*}
Interpolating $\bd$ once yields the coherent TDG procedure
\begin{equation}\label{check_dI}
    \check{\bd}(\cS) = \Bar{\bd}(\cS)=
    \begin{cases}
        0 & \text{if } |\cS|\leq 1 \\
        \Xdata & \text{if } |\cS|\geq 2
    \end{cases}
    \quad \text{for all $\cS \subseteq \cI$.}
\end{equation}
Maximizing $\bd$ yields the perfectly coherent TDG procedure
\begin{equation}\label{dI_hat}
    \tilde{\bd}(\cS) = 
    \begin{cases}
        0 &\text{if $|\cS| \leq 1$} \\
        \Xdata &\text{if $|\cS| = 2$} \\
        2 \Xdata &\text{if $|\cS| = 3$} 
    \end{cases}
    \quad \text{for all $\cS\subseteq\cI$}.
\end{equation} 
Therefore, $\tilde{\bd}$ uniformly improves upon $\check{\bd}$, since $\tilde{\bd}(\cS)=\check{\bd}(\cS)$ for $|\cS|\leq 2$ and
\begin{equation*}
    \P(\tilde{\bd} (\{1,2,3\}) >  \check{\bd} (\{1,2,3\})) = \alpha.
\end{equation*}
Derivations for $\check{\bd}$ and $\tilde{\bd}$ are provided in \appsref{showing_coherence} and \ref{app:max_dI}, respectively. 
\end{example}

\exref{dI} provides a helpful illustration of why maximization can improve upon interpolation. When $\Xdata = 1$, maximization infers from the three bounds
\begin{equation}\label{3_d}
\bd(\{1,2\})
=
\bd(\{1,3\})
=
\bd(\{2,3\})
=
1
\end{equation}
that the number of true discoveries among $\{1,2,3 \}$ is at least two, thereby yielding $\tilde{\bd}(\{1,2,3\})=2$; see \figref{venn-null-hypotheses} for an illustration. In contrast, interpolation can only combine two lower bounds, $\bd(\cU)$ and $\bd(\cS \backslash \cU)$, at a time when improving the bound for a given $\cS$. Consequently, it cannot exploit the joint logical implications of the three bounds in \eqref{3_d} to improve $\check{\bd}(\{1,2,3\}) = \Bar{\bd}(\{1,2,3\}) =1$.

\begin{figure}[t]
    \centering

    \begin{tikzpicture}[scale=0.8]

    \def\r{2}

    \coordinate (A) at (-1.2,0.6);
    \coordinate (B) at ( 1.2,0.6);
    \coordinate (C) at ( 0,-1.25);

    \draw[thick] (-4,-3.8) rectangle (4,3.2);
    \node[anchor=south east] at (3.8,-3.6) {$\cP$};

    \begin{scope}
        \clip (A) circle (\r);
        \clip (B) circle (\r);
        \fill[red!25] (-4,-4) rectangle (4,4);
    \end{scope}

    \begin{scope}
        \clip (A) circle (\r);
        \clip (C) circle (\r);
        \fill[red!25] (-4,-4) rectangle (4,4);
    \end{scope}

    \begin{scope}
        \clip (B) circle (\r);
        \clip (C) circle (\r);
        \fill[red!25] (-4,-4) rectangle (4,4);
    \end{scope}

    \begin{scope}
        \clip (A) circle (\r);
        \clip (B) circle (\r);
        \clip (C) circle (\r);
        \fill[red!55] (-4,-4) rectangle (4,4);
    \end{scope}

    \draw[thick] (A) circle (\r);
    \draw[thick] (B) circle (\r);
    \draw[thick] (C) circle (\r);

    \node at (-2.45,1.55) {$\cH_1$};
    \node at ( 2.55,1.55) {$\cH_2$};
    \node at ( 0,-2.75) {$\cH_3$};

\end{tikzpicture}

    \caption{Recall that each hypothesis $\cH_1,\cH_2,\cH_3$ corresponds to a set of distributions, as visualized by the Venn diagram. The shaded regions indicate where the true distribution $\P$ cannot lie given the bounds in \eqref{3_d}. If these bounds are correct, then $\P$ can belong to at most one of the three hypotheses. Equivalently, at least two of the three hypotheses must be false.}\label{fig:venn-null-hypotheses}
\end{figure}
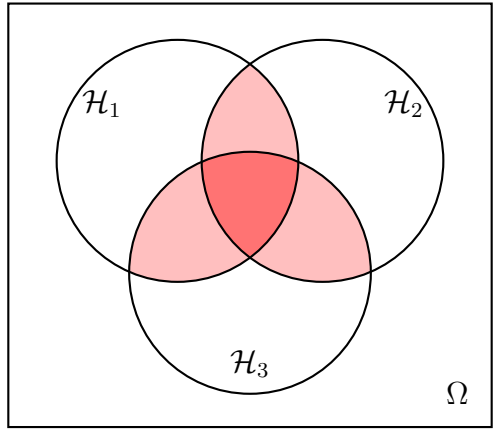

We close this section by highlighting an important connection with the work of \citet{blanchard2020post}, who develop procedures for upper bounding false discoveries from a \textit{reference family}. Our maximization operation is the true-discovery analogue of their \textit{optimal bound}, defined in equation (7) of \citet{blanchard2020post}, with both approaches leveraging all available bounds to obtain improved guarantees, but in opposite directions. For nested reference families, \citet[Proposition 2.5]{blanchard2020post} show that their interpolation bound coincides with the optimal bound. \citet[Theorem 3.6]{Durand2020} extend this equivalence to forest-structured reference families, in which any two sets are either nested or disjoint. Our work takes a different perspective by investigating the notion of coherence introduced by \citet{Goeman2021}, where interpolation is applied repeatedly until no further improvement is possible. We show that even after this process has reached a coherent procedure, further improvements may still be implied by its bounds and can be obtained through maximization.

\subsection{When does interpolation lead to perfect coherence?}\label{sec:TU_equivalence}
We now consider sufficient conditions under which applying interpolation until coherence is reached also yields a perfectly coherent procedure. An immediate case occurs when $\bd$ is already perfectly coherent, since interpolation cannot improve it. In the following, we provide a nontrivial condition that covers and extends the nested and forest-structured settings considered by \citet{blanchard2020post} and \citet{Durand2020}. Proofs of the results developed in this section are deferred to \appref{TU_equivalence}.

Define the support family of a TDG procedure $\bd$ to be
\begin{equation*}
    \cF (\bd) =  \{\cR \subseteq \cI:\bd(\cR) > 0\} = \{\cR_1,\dots,\cR_n\},
\end{equation*}
where $n=|\cF (\bd)|$. Let $\mathbf{M}\in\{0,1\}^{n\times|\cI|}$ denote an \textit{incidence} matrix for $\bd$ where the $(k,i)$th entry of $\mathbf{M}$ is
\begin{equation*}
    \mathbf{M}_{ki} = \mathbbm{1}\{i\in\cR_k\} \quad \text{for each $k\in\{1,\ldots,n\}$ and $i\in\cI$.}
\end{equation*}
Therefore, each row of $\mathbf{M}$ corresponds to a set in the support of $\bd$, while each column corresponds to a hypothesis. The following theorem gives a sufficient condition on $\mathbf{M}$ under which applying interpolation until coherence is reached also yields a perfectly coherent procedure, so that $\check{\bd}=\tilde{\bd}$.

\begin{theorem}\label{thm:TU_equivalence}
For a given TDG procedure $\bd$, suppose its incidence matrix $\mathbf{M}$ is totally unimodular almost surely. That is, every square submatrix of $\mathbf{M}$ has determinant in $\{-1,0,1\}$. Then
\begin{equation*}
    \check{\bd}(\cS) = \tilde{\bd}(\cS) \quad \text{for all $\cS \subseteq \cI$ almost surely.}
\end{equation*}
\end{theorem}

Following \citet{Durand2020}, we say that a support family is forest-structured if any two of its sets are either nested or disjoint; equivalently,
\begin{equation*}
    \cR\cap\cR'
    \in
    \{\cR,\cR',\varnothing\}
    \quad
    \text{for all $\cR,\cR'\in\cF(\bd)$}.
\end{equation*}
Such families have totally unimodular incidence matrices (\propref{forest_TU}), so \thmref{TU_equivalence} extends the forest-structured setting of \citet[Theorem 3.6]{Durand2020}, where their interpolation bound is shown to coincide with the optimal bound. Nested support families, as studied by \citet{blanchard2020post}, are a special case of forest structure, where
\begin{equation*}
    \cR_1\subsetneq \cR_2\subsetneq \cdots \subsetneq \cR_n.
\end{equation*}
These families arise naturally from procedures that threshold $p$-values or other test statistics, since the resulting rejection sets are nested as the threshold varies. Examples include the simultaneous FDP bounds of \citet{Hemerik2019} and the knockoff-based interpolation procedures of \citet{Li2024} and \citet{katsevich2020simultaneous}.
It should be noted that total unimodularity is more general than forest structure, allowing support families with more complex patterns of overlap. For example, if $\cI=\{1,2,3,4\}$ and
\begin{equation*}
    \cF(\bd)
    =
    \bigl\{ \{1,2\}, \{2,3\}, \{3,4\}
    \bigr\},
\end{equation*}
then the corresponding incidence matrix is totally unimodular, even though $\cF(\bd)$ is not forest-structured. 

\begin{proposition}\label{prop:forest_TU}
If the support family $\cF(\bd)$ is forest-structured, then the corresponding incidence matrix $\mathbf{M}$ is totally unimodular.
\end{proposition}

For a fixed discovery set $\cS$, a single interpolation update requires scanning the current support family, whereas direct maximization may require searching over the power set of $\cI$. Hence, individual interpolation updates can be substantially less computationally intensive. It may therefore be worthwhile to first check whether the incidence matrix of the support family is totally unimodular, which can be done in polynomial time using the algorithm of \citet{Truemper1990}. 
This condition may be particularly useful in settings such as the framework of \citet{Xu2024}, where a TDG procedure is built from prespecified \emph{focus sets}, which play a role analogous to the support family considered here. Their procedure repeatedly applies interpolation until coherence is reached, while allowing the focus sets to have an unrestricted, potentially overlapping structure. Our result therefore provides a sufficient condition under which this process exhausts the information in the initial bounds. In \secref{discussion}, we give a related meta-analysis example with a non-totally-unimodular support structure for which maximization yields strictly stronger bounds than repeated interpolation. 

\subsection{Closed testing procedures are perfectly coherent}\label{sec:closure}
Closed testing provides an important class of TDG procedures that, as we show, are automatically perfectly coherent. To establish this connection, we first recall the closed testing framework from Section 4 of \citet{Goeman2021}. For any finite subset $\cU \subseteq \cI$, let 
\begin{equation*}
\cH_{\cU} = \bigcap_{i \in \cU} \cH_{i}
\end{equation*}
denote an intersection hypothesis so that $\P \in \cH_{\cU}$ if and only if $\P \in \cH_i$ for every $i \in \cU$. Let $\phi_{\cU} \in \{0,1\}$ denote a \textit{local test} for $\cH_{\cU}$, where $\cH_{\cU}$ is rejected if $\phi_{\cU} = 1$ and not rejected if $\phi_{\cU} = 0$. Following \citet{Goeman2021}, we set $\phi_{\varnothing} = 0$ surely. The collection $\varphi = (\phi_{\cU})_{\cU \subseteq \cI}$ is called a \textit{suite} of local tests. A local test $\phi_{\cU}$ is said to be \textit{valid} if 
\begin{equation*}
    \P ( \phi_{\cU} = 1) \leq \alpha \quad \text{for all $\P \in \cH_{\cU}$.}
\end{equation*}
Let the \textit{effective} local test for $\cH_{\cV}$, where $\cV \subseteq \cI$, be
\begin{equation*}
    \phi^{\cI}_{\cV} = \min \{ \phi_{\cU} : \cV \subseteq \cU \subseteq \cI \},
\end{equation*}
which rejects $\cH_{\cV}$ if and only if $\phi_{\cU}=1$ for all $\cV \subseteq \cU \subseteq \cI$. The closed testing TDG procedure induced by a suite of valid local tests $\varphi$ is defined as
\begin{equation}\label{closed_test}
    \bd_{\varphi}(\cS) = \min_{\cU \subseteq \cS} \{ |\cS\backslash \cU|: \phi^{\cI}_{\cU} = 0 \} \quad \text{for all $\cS \subseteq \cI$}.
\end{equation}

Heuristically, closed testing procedures are perfectly coherent because the bound for each set $\cS$ already incorporates information from the effective tests of all its subsets $\cU\subseteq\cS$, each of which depends on the local tests $\phi_{\cV}$ for all $\cU\subseteq\cV\subseteq\cI$. In this sense, closed testing constructs bounds by drawing information jointly across all subsets of $\cI$.
\thmref{iff_perfect_coherence} formalizes this intuition and establishes the stronger result that a TDG procedure is perfectly coherent if and only if it is a closed testing procedure.

\begin{theorem}\label{thm:iff_perfect_coherence}
A TDG procedure is perfectly coherent if and only if it can be represented as a closed testing procedure induced by a suite of valid local tests.
\end{theorem}

The proof of \thmref{iff_perfect_coherence} is provided in \appref{closure}. Since the TDG procedure in \eqref{check_dI} of \exref{dI} is coherent but not perfectly coherent, \thmref{iff_perfect_coherence} implies that \eqref{check_dI} cannot be represented as a closed testing procedure. Consequently, Corollary 1 of \citet{Goeman2021} is incorrect, as it states that every coherent TDG procedure can be represented as a closed testing procedure. This does not, however, affect the validity of the other results in \citet{Goeman2021}. 

Aside from their Corollary 1, our result can be viewed as complementary to the work of \citet{Goeman2021}. For instance, Theorem 1 of \citet{Goeman2021} implies that every admissible TDG procedure is equivalent to a closed testing procedure. This aligns naturally with \thmref{iff_perfect_coherence}, since any admissible procedure must be perfectly coherent; otherwise, it could be uniformly improved by maximization. Since every perfectly coherent procedure is coherent, \thmref{iff_perfect_coherence} also recovers Lemma 4 of \citet{Goeman2021}, which states that closed testing procedures are coherent.

\section{Discussion}\label{sec:discussion}
We conclude by noting that there are other applications of our theory, besides multiple testing based on a single dataset. Namely, we highlight an application of maximization to meta-analysis.
To illustrate, let $\cI=\{1,2,3,4\}$ denote the index set of hypotheses. Consider four studies, each reporting a $0.95$ confidence lower bound on the number of true discoveries  within a subset of the same hypothesis family. Suppose these subsets are $\{1\}$, $\{2,3\}$, $\{1,2,4\}$, and $\{1,3,4\}$, respectively. Such confidence bounds could be obtained in several ways, for example by closed testing \citep{Goeman2011,Goeman2021}, or the resampling-based approach of \citet{Meinshausen2005}.
Suppose the reported bounds are
\begin{equation*}
\bd(\{1    \})=1, \quad
\bd(\{2,3  \})=1, \quad
\bd(\{1,2,4\})=2, \quad
\bd(\{1,3,4\})=2.
\end{equation*}

The particular pattern of overlap among the sets in the support family yields an incidence matrix that is not totally unimodular.
By Boole's inequality, 
the four bounds hold simultaneously with confidence at least
\begin{equation*}
    1-4(1-0.95)=0.80.
\end{equation*}
Therefore, $\bd$ defines a $0.80$ TDG procedure, with $\bd(\cS)=0$ for all other $\cS\subseteq\cI$. Applying maximization and, separately, interpolation, it can be verified that
\begin{equation*}
    2=\tilde{\bd}(\{2,3,4\})>\check{\bd}(\{2,3,4\})=1
    \quad\text{and}\quad
    3=\tilde{\bd}(\cI)>\check{\bd}(\cI)=2,
\end{equation*}
while $\tilde{\bd}(\cS)=\check{\bd}(\cS)$ for all other $\cS\subseteq\cI$. Hence, maximization allows the four study-specific confidence bounds to be combined to obtain simultaneous lower bounds for other subsets of hypotheses, while retaining confidence at least $0.80$. Moreover, two of the resulting bounds are strictly stronger than those obtained by interpolation until coherence. 



\bibliographystyle{plainnat}   
\bibliography{references}

\appendix

\section{Results relating to \exref{dI}}\label{app:dI}

\subsection{A partial conjunction testing interpretation of $\bd$ in \exref{dI}}\label{app:PC}
We show here that the construction of $\bd$ in \exref{dI} is motivated by a partial conjunction (PC) test. Let $\cI=\{1,2,3\}$, and let $P_1,P_2,P_3$ be valid $p$-values for testing $\cH_1,\cH_2,\cH_3$, respectively. 

For each pair $\{i,j\}\subseteq\cI$, a valid $p$-value for testing the intersection hypothesis $\cH_i\cap\cH_j$ is given by
\begin{equation*}
    P_{ij}
    =
    \min\{1,2\min(P_i,P_j)\}.
\end{equation*} 
Now consider the PC hypothesis
\begin{equation*}
    \cH^{\mathrm{PC}}
    =
    (\cH_1\cap\cH_2)
    \cup
    (\cH_1\cap\cH_3)
    \cup
    (\cH_2\cap\cH_3).
\end{equation*}
If $\cH^{\mathrm{PC}}$ is false, then every intersection hypothesis pair contains at least one false null. By equation (3) of \citet{Benjamini2008}, a valid $p$-value for testing $\cH^{\mathrm{PC}}$ is
\begin{equation*}
    P^{\mathrm{PC}}
    =
    \max_{ \{i,j\}\subseteq\cI } P_{ij}.
\end{equation*}
Define the test statistic
\begin{equation}\label{X}
    \Xdata
    =
    \mathbbm{1}\left\{
        P^{\mathrm{PC}}\leq\alpha
    \right\}
\end{equation}
where $\Xdata=1$ corresponds to rejecting $\cH^{\mathrm{PC}}$ at significance level $\alpha$. It follows that
\begin{equation}\label{PC_validity}
    \P(\Xdata=1)
    =
    \P(P^{\mathrm{PC}}\leq\alpha)
    \leq\alpha
    \quad
    \text{for all $\P\in\cH^{\mathrm{PC}}$.}
\end{equation}

Based on the test statistic $\Xdata$ defined in \eqref{X}, define the procedure
\begin{equation*}
    \bd(\cS)
    =
    \begin{cases}
        \Xdata & \text{if } |\cS|=2\\
        0 & \text{otherwise}
    \end{cases}
    \quad \text{for all $\cS\subseteq\cI$},
\end{equation*}
which is the form given in \exref{dI}. To verify its $(1-\alpha)$ guarantee, define the failure event
\begin{equation*}
    \cE
    =
    \left\{
        \exists\,\cS\subseteq\cI
        \text{ such that }
        \bd(\cS)>|\cS\cap\cI_1|
    \right\}.
\end{equation*}
If $\P\notin\cH^{\mathrm{PC}}$, then every $\cS$ with $|\cS|=2$ contains at least one false null. In this case, $\bd(\cS)\leq|\cS\cap\cI_1|$ for every $\cS\subseteq\cI$, and therefore $\P(\cE)=0$. On the other hand, if $\P\in\cH^{\mathrm{PC}}$, then $\cE$ can occur only when $\Xdata=1$. Hence, by \eqref{PC_validity}, we have that
\begin{equation*}
    \P(\cE)
    \leq
    \P(\Xdata=1)
    \leq
    \alpha.
\end{equation*}

Therefore, it holds for any $\P$ that
\begin{equation*}
    \P\left(
        \bd(\cS)\leq|\cS\cap\cI_1|
        \text{ for all $\cS\subseteq\cI$}
    \right)
    =
    1-\P(\cE)
    \geq
    1-\alpha,
\end{equation*}
so $\bd$ has a $(1-\alpha)$ guarantee.

\subsection{Coherence of $\check{\bd}$ in \exref{dI}}\label{app:showing_coherence}
We first determine the interpolation of $\bd$. On the event $\{\Xdata=0\}$, $\bd(\cW) = 0$ for all $\cW \subseteq \cI$, and therefore
\begin{equation*}
    \Bar{\bd}(\cS)
    =
    \max_{\cU\subseteq\cI}
    \left\{
        -|\cU\backslash\cS|
    \right\}
    =
    0  \quad \text{for all $\cS\subseteq\cI$}.
\end{equation*}
On the event $\{\Xdata=1\}$, \tabref{dI_interpolation} enumerates
$\bd(\cU) - |\cU\backslash\cS| + \bd(\cS\backslash\cU)$
for all $\cS,\cU\subseteq\cI$. It follows that
\begin{equation*}
    \Bar{\bd}(\cS)
    =
    \max_{\cU\subseteq\cI}
    \left\{
        \bd(\cU)
        -
        |\cU\backslash\cS|
        +
        \bd(\cS\backslash\cU)
    \right\}
    =
    \begin{cases}
        0 & \text{if } |\cS|\leq 1\\
        1 & \text{if } |\cS|\geq 2
    \end{cases}.
\end{equation*}
Combining the events $\{\Xdata=0\}$ and $\{\Xdata=1\}$ yields the interpolated procedure
\begin{equation*}
    \Bar{\bd}(\cS)
    =
    \begin{cases}
        0 & \text{if } |\cS|\leq 1\\
        \Xdata & \text{if } |\cS|\geq 2
    \end{cases}
    \quad \text{for all $\cS\subseteq\cI$}.
\end{equation*}

Let $\Bar{\bd}^{(2)}$ denote the procedure resulting from interpolating $\bar{\bd}$. 
On the event $\{\Xdata=0\}$, $\Bar{\bd}(\cW) = 0$ for all $\cW \subseteq \cI$, and therefore 
\begin{equation*}
    \Bar{\bd}^{(2)}(\cS) =    
    \max_{\cU\subseteq\cI}
    \left\{
        -|\cU\backslash\cS|
    \right\} = 0 \quad \text{for all $\cS\subseteq\cI$}.
\end{equation*}
On the event $\{\Xdata=1\}$, \tabref{dI_second_interpolation} enumerates
$\Bar{\bd}(\cU) - |\cU\backslash\cS| + \Bar{\bd}(\cS\backslash\cU)$
for all $\cS,\cU\subseteq\cI$. Consequently,
\begin{equation*}
    \Bar{\bd}^{(2)}(\cS)
    =
    \max_{\cU\subseteq\cI}
    \left\{
        \Bar{\bd}(\cU)
        -
        |\cU\backslash\cS|
        +
        \Bar{\bd}(\cS\backslash\cU)
    \right\}
    =
    \begin{cases}
        0 & \text{if } |\cS|\leq 1 \\
        1 & \text{if } |\cS|\geq 2
    \end{cases}
    \quad \text{for all $\cS\subseteq\cI$}.
\end{equation*}
Combining the events $\{\Xdata=0\}$ and $\{\Xdata=1\}$ yields the double interpolated procedure
\begin{equation*}
    \Bar{\bd}^{(2)}(\cS)
    =
    \begin{cases}
        0 & \text{if } |\cS|\leq 1\\
        \Xdata & \text{if } |\cS|\geq 2
    \end{cases}
    \quad \text{for all $\cS\subseteq\cI$}.
\end{equation*} Thus $\Bar{\bd}^{(2)}=\Bar{\bd}$, so $\Bar{\bd}$ is coherent and $\check{\bd}=\Bar{\bd}$.
\vspace{1em}

\begin{table}[ht]
\centering
\scriptsize
\setlength{\tabcolsep}{4pt}
\renewcommand{\arraystretch}{1.25}
\begin{tabular}{c|cccccccc}
\diagbox{$\cS$}{$\cU$}
& $\varnothing$
& $\{1\}$
& $\{2\}$
& $\{3\}$
& $\{1,2\}$
& $\{1,3\}$
& $\{2,3\}$
& $\{1,2,3\}$ \\
\hline
$\varnothing$
& $0$ & $-1$ & $-1$ & $-1$ & $-1$ & $-1$ & $-1$ & $-3$ \\

$\{1\}$
& $0$ & $0$ & $-1$ & $-1$
& $0$ & $0$ & $-1$ & $-2$ \\

$\{2\}$
& $0$ & $-1$ & $0$ & $-1$
& $0$ & $-1$ & $0$ & $-2$ \\

$\{3\}$
& $0$ & $-1$ & $-1$ & $0$
& $-1$ & $0$ & $0$ & $-2$ \\

$\{1,2\}$
& $1$ & $0$ & $0$ & $0$
& $1$ & $0$ & $0$ & $-1$ \\

$\{1,3\}$
& $1$ & $0$ & $0$ & $0$
& $0$ & $1$ & $0$ & $-1$ \\

$\{2,3\}$
& $1$ & $0$ & $0$ & $0$
& $0$ & $0$ & $1$ & $-1$ \\

$\{1,2,3\}$
& $0$ & $1$ & $1$ & $1$
& $1$ & $1$ & $1$ & $0$
\end{tabular}
\vspace{0.5em}
\caption{Values of
$\bd(\cU)-|\cU\backslash\cS|+\bd(\cS\backslash\cU)$
on the event $\{\Xdata=1\}$ for all $\cS,\cU\subseteq\cI$.}
\label{tab:dI_interpolation}
\end{table}

\begin{table}[ht]
\centering
\scriptsize
\setlength{\tabcolsep}{4pt}
\renewcommand{\arraystretch}{1.25}
\begin{tabular}{c|cccccccc}
\diagbox{$\cS$}{$\cU$}
& $\varnothing$
& $\{1\}$
& $\{2\}$
& $\{3\}$
& $\{1,2\}$
& $\{1,3\}$
& $\{2,3\}$
& $\{1,2,3\}$ \\
\hline
$\varnothing$
& $0$ & $-1$ & $-1$ & $-1$
& $-1$ & $-1$ & $-1$ & $-2$ \\

$\{1\}$
& $0$ & $0$ & $-1$ & $-1$
& $0$ & $0$ & $-1$ & $-1$ \\

$\{2\}$
& $0$ & $-1$ & $0$ & $-1$
& $0$ & $-1$ & $0$ & $-1$ \\

$\{3\}$
& $0$ & $-1$ & $-1$ & $0$
& $-1$ & $0$ & $0$ & $-1$ \\

$\{1,2\}$
& $1$ & $0$ & $0$ & $0$
& $1$ & $0$ & $0$ & $0$ \\

$\{1,3\}$
& $1$ & $0$ & $0$ & $0$
& $0$ & $1$ & $0$ & $0$ \\

$\{2,3\}$
& $1$ & $0$ & $0$ & $0$
& $0$ & $0$ & $1$ & $0$ \\

$\{1,2,3\}$
& $1$ & $1$ & $1$ & $1$
& $1$ & $1$ & $1$ & $1$
\end{tabular}
\vspace{0.5em}
\caption{Values of
$\Bar{\bd}(\cU)-|\cU\backslash\cS|+\Bar{\bd}(\cS\backslash\cU)$
on the event $\{\Xdata=1\}$ for all $\cS,\cU\subseteq\cI$.}
\label{tab:dI_second_interpolation}
\end{table}

\subsection{Maximization of $\bd$ in \exref{dI}}\label{app:max_dI}
When $\Xdata=0$, we have $\bd(\cU)=0$ for all $\cU \subseteq \cI$. Hence $\cA = 2^{\cI}$ and
\begin{equation*}
\tilde{\bd}(\cS)
=
\min_{\cJ \in 2^{\cI}}|\cJ \cap \cS| = |\varnothing \cap \cS| 
=
0
\quad\text{for all $\cS \subseteq \cI.$}
\end{equation*}
When $\Xdata=1$, a direct check gives
$
\cA 
=
\{
\{1,2\},
\{1,3\},
\{2,3\},
\{1,2,3\}
\}$,
from which
\begin{equation*}
\tilde{\bd}(\cS) = \min_{\cJ \in \cA}|\cJ \cap \cS|
=
\begin{cases}
0 & \text{ if } |\cS|\leq 1\\
1 & \text{ if } |\cS|=2\\
2 & \text{ if } |\cS|=3
\end{cases}
\quad\text{for all $\cS \subseteq \cI.$}
\end{equation*}
Hence, \eqref{dI_hat} is indeed the maximization of $\bd$.

\section{Proofs of \lemref{d_hat_improve} and \propref{max_is_max}}\label{app:perfect_coherence}

\subsection{Proof of \lemref{d_hat_improve}}
\begin{proof}[\unskip\nopunct]
We first demonstrate that $\tilde{\bd}$ has a $(1 - \alpha)$ guarantee. Define the event 
    \begin{equation*}
        \cE = \{ \bd(\cS) \leq |\cI_1 \cap \cS| \text{ for all $\cS \subseteq \cI$} \},
    \end{equation*}
    where $\P(\cE) \geq 1-\alpha$ for all $\P \in \cP$ by the $(1-\alpha)$ guarantee of $\bd$. If $\cE$ happened, then 
    \begin{equation*}
        \tilde{\bd}(\cS) = \min_{\cJ \in \cA}|\cJ \cap \cS| \leq |\cI_1 \cap \cS| \quad \text{for all $\cS \subseteq \cI$}
    \end{equation*}
    by the formulation of $\cA$. Hence, for all $\P \in \cP$,
    \begin{equation*}
        \P \left( \tilde{\bd}(\cS) \leq |\cI_1 \cap \cS| \text{ for all $\cS \subseteq \cI$} \right) \geq  \P(\cE) \geq 1-\alpha.
    \end{equation*}
Hence, $\tilde{\bd}$ does indeed have a $(1 - \alpha)$ guarantee.
    
Next, we show that $\check{\bd}(\cS) \leq \tilde{\bd}(\cS)$ for all $\cS \subseteq \cI$. Fix any $\cS \subseteq \cI$.
Let $\bd^{(0)}=\bd$, and let $\bd^{(k)}$ denote the procedure obtained after $k$ rounds of interpolation. We show by induction that the statement:
\begin{equation}\label{interpolation_feasible}
    \bd^{(k)}(\cS)\leq |\cJ\cap\cS|
    \quad\text{for all $\cS\subseteq\cI$ and $\cJ\in\cA$,}
\end{equation}
holds for all $k \geq 0$. 

For $k=0$, \eqref{interpolation_feasible} follows directly from the formulation of $\cA$. Now suppose \eqref{interpolation_feasible} holds for some $k$ and fix any $\cS\subseteq\cI$ and $\cJ\in\cA$. Then, for any $\cU \subseteq \cI$,
\begin{align}
    \bd^{(k)}(\cU) - |\cU \backslash \cS| +\bd^{(k)}(\cS\setminus\cU) &\leq |\cJ \cap \cU| - |\cU \backslash \cS| + |\cJ \cap (\cS \backslash \cU )| \nonumber \\
    &\leq |\cJ\cap (\cU \cap \cS) |+|\cJ\cap(\cS\setminus\cU)|  \nonumber \\
    &=|\cJ\cap\cS|.  \label{dk}
\end{align}
By the interpolation formulation, we have that
\begin{equation}\label{dkplus1}
    \bd^{(k+1)}(\cS) = \max_{\cU \subseteq \cI} \left\{  \bd^{(k)}(\cU) - |\cU \backslash \cS| + \bd^{(k)}(\cS\setminus\cU) \right\}.
\end{equation}
Combining \eqref{dk} and \eqref{dkplus1} yields $\bd^{(k+1)}(\cS)\leq|\cJ\cap\cS|$. Hence, \eqref{interpolation_feasible} holds true for all $k > 0$ by induction.

Since $\check{\bd}$ is obtained after finitely many rounds of interpolation, it follows from \eqref{interpolation_feasible} that
\begin{equation*}
    \check{\bd}(\cS)\leq|\cJ\cap\cS|
     \quad\text{for all $\cS\subseteq\cI$ and $\cJ\in\cA$.}
\end{equation*}
Therefore, it holds that
\begin{equation*}
    \check{\bd}(\cS)
    \leq
    \min_{\cJ\in\cA}|\cJ\cap\cS|
    =
    \tilde{\bd}(\cS) \quad \text{for all $\cS\subseteq\cI$.}
\end{equation*}
\end{proof}

\subsection{Proof of \propref{max_is_max}}
\begin{proof}[\unskip\nopunct]

    Let $\tilde{\bd}^{(2)}$ denote the maximization of $\tilde{\bd}$, where $\tilde{\bd}$ is the maximization of $\bd$. Then for all $\cS \subseteq \cI$,
    \begin{equation*}
        \tilde{\bd}^{(2)}(\cS)= \min_{J \in \tilde{\cA}} |\cJ \cap \cS| \quad \text{where} \quad
        \tilde{\cA} = \cA(\tilde{\bd}) .
    \end{equation*}
    To prove \propref{max_is_max}, i.e., that $\tilde{\bd}^{(2)} = \tilde{\bd}$, it suffices to show that $\tilde{\cA} = \cA$. 
    
    For any given $\cJ \in \cA$, 
     \begin{equation*}
         \tilde{\bd}(\cS) = \min_{\cV \in \cA}|\cV \cap \cS| \leq |\cJ \cap \cS| \quad \text{for all $\cS \subseteq \cI$.}
     \end{equation*}
     Hence, $\cA \subseteq \tilde{\cA}$ by the formulation of $\tilde{\cA}$. Now it remains to show that $\tilde{\cA} \subseteq \cA$.
     
    For any $\cJ \in \tilde{\cA}$, we have that
    \begin{equation*}
       \bd(\cS) \leq \tilde{\bd}(\cS) \leq \tilde{\bd}^{(2)}(\cS)  \leq |\cJ \cap \cS| \quad \text{for all $\cS \subseteq \cI$}.
    \end{equation*}
    Hence, $\tilde{\cA} \subseteq \cA$ by the formulation of $\cA$.
\end{proof}

\section{Results and proofs relating to \thmref{TU_equivalence} and \propref{forest_TU}}\label{app:TU_equivalence}
\subsection{Technical results for proving \thmref{TU_equivalence}}
\begin{lemma}\label{lem:binary_representation}
Fix any $\cS\subseteq\cI$. For each $\cJ\subseteq\cI$, define the indicator variable
\begin{equation*}
    w_i(\cJ)=\mathbbm{1}\{i\in\cJ\} \quad \text{for each $i\in\cS$.}
\end{equation*}
Then 
\begin{equation*}
    \left\{
        (w_i(\cJ))_{i\in\cS} :
        \cJ\subseteq\cI,\ 
        \cI\backslash\cS\subseteq\cJ
    \right\}
    =
    \{0,1\}^{|\cS|}.
\end{equation*}
\end{lemma}

\begin{proof}[Proof of \lemref{binary_representation}]
For any $\cJ\subseteq\cI$, each $w_i(\cJ)$ is a value in $\{0,1\}$. Hence,
\begin{equation*}
    \left\{
        (w_i(\cJ))_{i\in\cS} :
        \cJ\subseteq\cI,\ 
        \cI\backslash\cS\subseteq\cJ
    \right\}
    \subseteq
    \{0,1\}^{|\cS|}.
\end{equation*}
It now remains to prove the reverse inclusion. Fix any binary vector 
\begin{equation*}
    \mathbf{x}=(x_i)_{i\in\cS}\in\{0,1\}^{|\cS|}
\end{equation*}
and define the set
\begin{equation*}
    \cJ_{\mathbf{x}}
    =
    (\cI\backslash\cS)
    \cup
    \{i\in\cS:x_i=1\}.
\end{equation*}
Then $\cJ_{\mathbf{x}}\subseteq\cI$ and $\cI\backslash\cS\subseteq\cJ_{\mathbf{x}}$ by construction. Moreover, for each $i\in\cS$,
\begin{equation*}
    w_i(\cJ_{\mathbf{x}})
    =
    \mathbbm{1}\{i\in\cJ_{\mathbf{x}}\}
    =
    x_i.
\end{equation*}
Hence $(w_i(\cJ_{\mathbf{x}}))_{i\in\cS}= \mathbf{x}$. Since $\mathbf{x} \in\{0,1\}^{|\cS|}$ is arbitrary and 
\begin{equation*}
    (w_i(\cJ_{\mathbf{x}}))_{i\in\cS} \in \left\{
    (w_i(\cJ))_{i\in\cS} :
    \cJ\subseteq\cI,\ 
    \cI\backslash\cS\subseteq\cJ
    \right\},
\end{equation*}
the reverse inclusion is proven.
\end{proof}

\begin{lemma}\label{lem:for_programming_representation}
    Fix any $\cS\subseteq\cI$. For the TDG procedure $d$ defined in \thmref{TU_equivalence}, define 
\begin{equation*}
    q_k = \bd(\cR_k)-|\cR_k\backslash\cS| \quad \text{for $k = 1,\dots,n$.}
\end{equation*}
Then there exists a non-negative integer $y_k^{\star} \in \mathbb{Z}_+$ for $k = 1,\dots,n$ such that
\begin{equation*}
    \tilde{\bd}(\cS) = \sum_{k=1}^n y_k^{\star}q_k - \sum_{i\in\cS}(r_i-1)_+,
\end{equation*}
where $r_i = \sum_{k:i\in\cR_k} y_k^{\star}$ and $(r_i-1)_+ = (r_i-1) \vee 0$ for each $i \in \cS$.

\end{lemma}
\begin{proof}[Proof of \lemref{for_programming_representation}]
    First, note that adding the elements of $\cI \backslash \cS$ to a candidate set $\cJ \in \cA$ does not change $|\cJ\cap\cS|$ and does not violate the constraint $\bd(\cR)\leq|\cJ\cap\cR|$ for all $\cR \subseteq \cI$. Hence,
\begin{align}
    \tilde{\bd}(\cS)
    &= \min_{\cJ \in \cA} \left\{ |\cJ\cap\cS| \right\} \nonumber \\
    &=
    \min_{\substack{\cJ\in \cA\\
         \cI\backslash\cS\subseteq\cJ}}
    \left\{ |\cJ\cap\cS| \right\} \nonumber \\
    &=
   \min_{\substack{\cJ\subseteq\cI\\
         \cI\backslash\cS\subseteq\cJ}}
    \left\{
        |\cJ\cap\cS|:
        \bd(\cR)\leq|\cJ\cap\cR|
        \text{ for all }\cR\subseteq\cI
    \right\} \nonumber \\
    &=
    \min_{\substack{\cJ\subseteq\cI\\
         \cI\backslash\cS\subseteq\cJ}}
    \left\{
        |\cJ\cap\cS|:
        \bd(\cR_k)\leq|\cJ\cap\cR_k|
        \text{ for $k = 1,\dots,n$} 
    \right\}, \label{max_d_ISJ}
\end{align}
where the last equality follows because $\bd(\cR)=0$ for every $\cR\notin\mathcal F(\bd)$.  

Consider any $\cJ \subseteq \cI$ such that $\cI\backslash\cS\subseteq\cJ$. Since $\cR_k$ can be expressed as
\begin{equation*}
    \cR_k  = (\cR_k \cap \cS) \cup (\cR_k \backslash \cS)
\end{equation*}
for $k = 1,\dots,n$, it follows that
\begin{align}\label{full_JR}
    |\cJ \cap \cR_k| = |\cJ \cap \cR_k \cap \cS| + |\cJ \cap (\cR_k \backslash \cS)|.
\end{align}
Moreover, since $(\cR_k \backslash \cS) \subseteq (\cI \backslash \cS) \subseteq \cJ$, we have that
\begin{equation}\label{JcapRbackS}
    \cJ \cap (\cR_k \backslash \cS)
    =
    \cR_k \backslash \cS.
\end{equation}
Combining \eqref{full_JR} and \eqref{JcapRbackS}, we get
\begin{equation*}
    |\cJ \cap \cR_k|
    =
    |\cJ \cap \cR_k \cap \cS|
    +
    |\cR_k \backslash \cS|.
\end{equation*}
Hence, \eqref{max_d_ISJ} is equivalent to
\begin{equation}\label{for_programming_representation}
    \tilde{\bd}(\cS)
    =
    \min_{\substack{\cJ\subseteq\cI\\
        \cI\backslash\cS\subseteq\cJ}}
    \left\{
        |\cJ\cap\cS|:
        \bd(\cR_k)
        \leq
        |\cJ\cap\cR_k\cap\cS|
        +
        |\cR_k\backslash\cS| \text{ for $k = 1,\dots,n$}
    \right\}.
\end{equation}

For a given candidate set $\cU \subseteq \cI$, define the indicator variable
\begin{equation*}
    w_i(\cU) = \mathbbm{1}\{i\in\cU\} \quad \text{for each $i\in\cS$.}
\end{equation*}
Moreover, let
\begin{equation*}
    \mathbf{x}=(x_i)_{i\in\cS}\in\{0,1\}^{|\cS|}
\end{equation*}
denote an arbitrary binary vector of length $|\cS|$. Then it follows from \eqref{for_programming_representation} and \lemref{binary_representation} that
\begin{align}
    \tilde{\bd}(\cS)
    &=
    \min_{\substack{\cJ\subseteq\cI\\
        \cI\backslash\cS\subseteq\cJ}}
    \left\{
        \sum_{i\in\cS} w_i(\cJ):
        \bd(\cR_k)
        \leq
        \sum_{i\in\cR_k\cap\cS} w_i(\cJ)
        +
        |\cR_k\backslash\cS| \text{ for $k = 1,\dots,n$} 
    \right\} \nonumber \\
    &=
    \min_{\mathbf{x} \in\{0,1\}^{|\cS|}}
    \left\{
        \sum_{i\in\cS}x_i:
        \sum_{i\in\cR_k\cap\cS}x_i
        \geq
        \bd(\cR_k)-|\cR_k\backslash\cS| \text{ for $k = 1,\dots,n$} 
    \right\} \nonumber 
\end{align}
Let $\mathbf{q}=(q_1,\ldots,q_n)^\top$ where $ q_k = \bd(\cR_k)-|\cR_k\backslash\cS|$ and $\mathbf{M}_{\cS}$ denote the $n\times|\cS|$ submatrix of $\mathbf{M}$ containing the columns indexed by $\cS$. Moreover, let $\mathbf{0}_{|\cS|}$ and $\mathbf{1}_{|\cS|}$ denote vectors of all zeros and ones of length $|\cS|$, respectively, and let $\mathbf{I}_{|\cS|}$ denote the $|\cS|\times|\cS|$ identity matrix. Then it follows that
\begin{align}
    \tilde{\bd}(\cS) &= \min
    \left\{
        \mathbf{1}_{|\cS|}^\top\mathbf{x}:
        \mathbf{M}_{\cS} \, \mathbf{x}\geq\mathbf{q},\
        \mathbf{0}_{|\cS|}\leq\mathbf{x}\leq\mathbf{1}_{|\cS|},\
        \mathbf{x}\in\mathbb{Z}^{|\cS|}
    \right\} \nonumber \\
    &= \min
    \left\{
        \mathbf{1}_{|\cS|}^\top\mathbf{x}:
        \begin{bmatrix}
            \mathbf{M}_{\cS} \\
            -\mathbf{I}_{|\cS|}
        \end{bmatrix}
        \mathbf{x}
        \geq
        \begin{bmatrix}
            \mathbf{q}\\
            -\mathbf{1}_{|\cS|}
        \end{bmatrix}, \
        \mathbf{x}\geq\mathbf{0}_{|\cS|}, \
        \mathbf{x}\in\mathbb{Z}^{|\cS|} 
    \right\}. \label{ILP}
\end{align}
Since $\mathbf{M}_{\cS}$ is a submatrix of the totally unimodular matrix $\mathbf{M}$, it is also totally unimodular. Moreover, adjoining rows of $-\mathbf{I}_{|\cS|}$ preserves total unimodularity. Since $\mathbf{q}$ and $\mathbf{1}_{|\cS|}$ are integer-valued, the integrality theorem for totally unimodular linear programs \citep[Corollary 8.1b]{Schrijver2004} implies that the linear programming relaxation of \eqref{ILP} has an integer-valued optimal solution. In other words, the constraint $\mathbf{x} \in \mathbb{Z}^{|\cS|}$ in \eqref{ILP} is redundant. Therefore,
\begin{equation}
    \tilde{\bd}(\cS)
    =
    \min
    \left\{
        \mathbf{1}_{|\cS|}^\top \mathbf{x}:
        \begin{bmatrix}
            \mathbf{M}_{\cS} \\
            -\mathbf{I}_{|\cS|}
        \end{bmatrix}
        \mathbf{x}
        \geq
        \begin{bmatrix}
            \mathbf{q}\\
            -\mathbf{1}_{|\cS|}
        \end{bmatrix}, \
        \mathbf{x}\geq\mathbf{0}_{|\cS|}
    \right\}.
    \label{ILP1}
\end{equation}
If we take \eqref{ILP1} as the primal problem, then its dual is given by
\begin{equation}
    \tilde{\bd}(\cS)
    =
    \max
    \left\{
        \begin{bmatrix}
            \mathbf{q}\\
            -\mathbf{1}_{|\cS|}
        \end{bmatrix}^{\top}\mathbf{u}:
        \begin{bmatrix}
            \mathbf{M}_{\cS}\\
            -\mathbf{I}_{|\cS|}
        \end{bmatrix}^{\top}
        \mathbf{u}
        \leq
        \mathbf{1}_{|\cS|}, \
        \mathbf{u}\geq\mathbf{0}_{n+|\cS|}
    \right\}
    \label{ILP2}
\end{equation}
where $\mathbf{u}$ denotes a vector of length $n + |\cS|$. Specifically, let $\mathbf{u}=(\mathbf{y},\mathbf{z})$, where
$\mathbf{y} = (y_k)^n_{k = 1} \in\mathbb{R}^n$ and
$\mathbf{z} = (z_i)_{i \in \cS} \in\mathbb{R}^{|\cS|}$. Then \eqref{ILP2} becomes
\begin{equation}\label{ILP3}
    \tilde{\bd}(\cS)
    =
    \max
    \left\{
        \mathbf{q}^\top\mathbf{y}
        -
        \mathbf{1}_{|\cS|}^\top\mathbf{z}:
        \mathbf{M}_{\cS}^\top\mathbf{y}
        -
        \mathbf{z}
        \leq
        \mathbf{1}_{|\cS|}, \
        \mathbf{y}\geq\mathbf{0}_n, \
        \mathbf{z}\geq\mathbf{0}_{|\cS|}
    \right\}.
\end{equation}
Since total unimodularity is preserved under transposition and $\mathbf{1}_{|S|}$ is integer-valued, the same integrality theorem implies that the dual problem \eqref{ILP2} has an integer-valued optimal solution. Therefore, $\tilde{d}(S)$ can be expressed as
\begin{equation}\label{ILP4}
    \tilde{\bd}(\cS) = \mathbf{q}^\top\mathbf{y}^{\star}
        - \mathbf{1}_{|\cS|}^\top\mathbf{z}^{\star}
\end{equation}
where
\begin{equation*}
    \mathbf{y}^{\star}= (y^\star_k)^n_{k = 1} \in\mathbb{Z}_+^n
    \quad\text{and}   \quad
    \mathbf{z}^{\star} = (z^\star_i)_{i \in \cS} \in\mathbb{Z}_+^{|\cS|}.
\end{equation*}

For each $i\in\cS$, the constraints in \eqref{ILP3} imply that
\begin{equation*}
    \sum_{k:i\in\cR_k} y_k - z_i\leq1,
\end{equation*}
or equivalently, 
\begin{equation}\label{z_constraint}
    z_i\geq \left( \sum_{k:i\in\cR_k} y_k -1 \right)_+,
\end{equation}
since $z_i$ is a non-negative integer. Given the form of $\tilde{\bd}(\cS)$ in \eqref{ILP4}, it must follow from \eqref{z_constraint} that the dual objective in \eqref{ILP3} is maximized by taking
\begin{equation*}
    z_i^{\star}=(r_i-1)_+ \quad \text{for each $i\in\cS$}
\end{equation*}
where $r_i = \sum_{k:i\in\cR_k} y_k^{\star}$. Therefore, \eqref{ILP4} can be written as
\begin{align*}
    \tilde{\bd}(\cS)
    =
    \mathbf{q}^\top\mathbf{y}^{\star}
    -
    \mathbf{1}_{|\cS|}^\top\mathbf{z}^{\star} 
    =
    \sum_{k=1}^n y_k^{\star}q_k
    -
    \sum_{i\in\cS}(r_i-1)_+.
\end{align*}
\end{proof}

\begin{lemma}\label{lem:coherent_aggregation}
For any given $\cS\subseteq\cI$, let $\cB_1,\ldots,\cB_n\subseteq\cS$ be $n$ arbitrary subsets of $\cS$. Moreover, for an arbitrary vector of non-negative integers $\mathbf{a}=(a_1,\ldots,a_n)^\top\in\mathbb{Z}_+^n$, define
\begin{equation*}
    \cU_{\mathbf{a}} = \bigcup_{k:a_k>0}\cB_k \quad \text{and} \quad r_{\mathbf{a},i} = \sum_{k:i\in\cB_k}a_k \quad \text{for each $i\in\cS$.}
\end{equation*}
If $\check{\bd}$ is a coherent TDG procedure, then
\begin{equation}\label{coherent_aggregation}
    \check{\bd}(\cU_{\mathbf{a}})
    \geq
    \sum_{k=1}^n a_k\check{\bd}(\cB_k)
    -
    \sum_{i\in\cS}(r_{\mathbf{a},i}-1)_+.
\end{equation}
\end{lemma}

\begin{proof}[Proof of \lemref{coherent_aggregation}]
Fix any $\cS \subseteq \cI$ and $\cB_1,\ldots,\cB_n\subseteq\cS$. Let the sum of the elements in $\mathbf{a}=(a_1,\ldots,a_n)^\top\in\mathbb{Z}_+^n$ be denoted as
\begin{equation*}
    L=\sum_{k=1}^n a_k.
\end{equation*}
We will prove the lemma by showing that \eqref{coherent_aggregation} holds for $L = 0,1,2,\dots$ via induction. For the initial case $L=0$, we must have $a_k=0$ for $k = 1,\dots,n$, so $\cU_{\mathbf{a}}=\varnothing$ and $r_{\mathbf{a},i}=0$ for every $i\in\cS$. Hence, the left- and right-hand sides of \eqref{coherent_aggregation} are zero, and the inequality in \eqref{coherent_aggregation} therefore must hold. 

We now make the inductive assumption that \eqref{coherent_aggregation} holds for every $\mathbf{a} \in \mathbb{Z}_+^n$ whose elements sum to $L-1$, where $L\geq1$. Consider an arbitrary vector of non-negative integers
\begin{equation}\label{b_vec}
    \mathbf{b}=(b_1,\ldots,b_n)^\top\in\mathbb{Z}_+^n \quad \text{such that} \quad \sum_{k=1}^n b_k=L.
\end{equation}
Fix a $j \in \{1, \dots, n\}$ such that $b_j\geq1$, and define
$\mathbf{a}'=(a_1',\ldots,a_n')^\top$ by
\begin{equation*}
    a_j'=b_j-1
    \qquad\text{and}\qquad
    a_k'=b_k
    \quad\text{for $k\neq j$.}
\end{equation*}
Then the sum of the elements in $\mathbf{a}'$ is
\begin{equation}\label{a_prime_sum}
    \sum_{k=1}^n a_k'=L-1.
\end{equation}
Now define
\begin{equation}\label{primes}
    \cU_{\mathbf{a}'} = \bigcup_{k:a_k'>0}\cB_k \quad \text{where} \quad r_{\mathbf{a}',i} = \sum_{k:i\in\cB_k}a_k' \quad \text{for each $i \in \cS$.}
\end{equation}
By the inductive assumption, it must follow from \eqref{a_prime_sum} that
\begin{equation}\label{aggregation_induction}
    \check{\bd}(\cU_{\mathbf{a}'})
    \geq
    \sum_{k=1}^n a_k'\check{\bd}(\cB_k)
    -
    \sum_{i\in\cS}(r_{\mathbf{a}',i}-1)_+.
\end{equation}

Define the set
\begin{equation}\label{U_equality}
    \cU_{\mathbf{b}} =
    \bigcup_{k:b_k>0}\cB_k =
    \cU_{\mathbf{a}'}\cup\cB_j,
\end{equation}
where the second equality follows from the formulation of $\mathbf{a}'$ and $\cU_{\mathbf{a}'}$.
By Lemma 3 of \citet{Goeman2021}, coherence implies for any
disjoint $\cV,\cW\subseteq\cI$, that
\begin{equation*}
    \check{\bd}(\cV)+\check{\bd}(\cW)
    \leq
    \check{\bd}(\cV\cup\cW)
    \leq
    \check{\bd}(\cV)+|\cW|.
\end{equation*}
Applying the first inequality to the disjoint sets $\cU_{\mathbf{a}'}$ and $\cB_j\backslash\cU_{\mathbf{a}'}$, and the second inequality to the disjoint sets $\cB_j\backslash\cU_{\mathbf{a}'}$ and $\cU_{\mathbf{a}'} \cap \cB_j$, gives
\begin{equation*}
    \check{\bd}(\cU_{\mathbf{a}'})+\check{\bd}(\cB_j\backslash\cU_{\mathbf{a}'})
    \leq
    \check{\bd}(\cU_{\mathbf{a}'}\cup\cB_j) \quad \text{and} \quad
    \check{\bd}(\cB_j) \leq \check{\bd}(\cB_j\backslash\cU_{\mathbf{a}'})+|\cU_{\mathbf{a}'} \cap \cB_j|,
\end{equation*}
respectively. These two inequalities imply that
\begin{equation}\label{coherence_inequality}
    \check{\bd}(\cU_{\mathbf{a}'}\cup\cB_j)
    \geq
    \check{\bd}(\cU_{\mathbf{a}'}) + \check{\bd}(\cB_j) -|\cU_{\mathbf{a}'}\cap\cB_j|.
\end{equation}

For the vector $\mathbf{b}$, define for each $i \in \cS$,
\begin{equation*}
    r_{\mathbf{b},i}
    =
    \sum_{k:i\in\cB_k}b_k. 
\end{equation*}
Since $\mathbf{b}$ is obtained from $\mathbf{a}'$ by adding one to its $j$th element, we have that
\begin{equation*}
    r_{\mathbf{b},i} = r_{\mathbf{a}',i}+\mathbbm{1}\{i\in\cB_j\}.
\end{equation*}
Therefore, it follows that
\begin{align*}
    (r_{\mathbf{b},i}-1)_+-(r_{\mathbf{a}',i}-1)_+
    &=
    \begin{cases}
        0
        & \text{if $i\notin\cB_j$}\\
        0
        & \text{if $i\in\cB_j$ and $r_{\mathbf{a}',i}=0$}\\
        1
        & \text{if $i\in\cB_j$ and $r_{\mathbf{a}',i}>0$}
    \end{cases}.
\end{align*}
By \eqref{primes}, it holds that $r_{\mathbf{a}',i}>0$ if and only if $i\in\cU_{\mathbf{a}'}$. This equivalence therefore gives
\begin{equation*}
    (r_{\mathbf{b},i}-1)_+-(r_{\mathbf{a}',i}-1)_+
    =
    \mathbbm{1}\{ \text{$i\in\cB_j$ and $r_{\mathbf{a}',i}>0$} \} =
    \mathbbm{1}\{i\in\cU_{\mathbf{a}'}\cap\cB_j\}.
\end{equation*}
Hence, $|\cU_{\mathbf{a}'}\cap\cB_j|$ can be expressed as
\begin{equation}\label{U_prime_cap_Bj}
    |\cU_{\mathbf{a}'}\cap\cB_j|
    =
    \sum_{i\in\cS}
    \left(
        (r_{\mathbf{b},i}-1)_+-(r_{\mathbf{a}',i}-1)_+
    \right).
\end{equation}

Combining \eqref{U_equality} and \eqref{coherence_inequality}, we get
\begin{align*}
    \check{\bd}(\cU_{\mathbf{b}})
    &\geq
    \check{\bd}(\cU_{\mathbf{a}'}) + \check{\bd}(\cB_j) -|\cU_{\mathbf{a}'}\cap\cB_j| \\
    &\geq
    \left( \sum_{k=1}^n a_k'\check{\bd}(\cB_k) - \sum_{i\in\cS}(r_{\mathbf{a}',i}-1)_+ \right)
    +
    \check{\bd}(\cB_j)
    -
    \sum_{i\in\cS}
    \left(
        (r_{\mathbf{b},i}-1)_+-(r_{\mathbf{a}',i}-1)_+
    \right) \\
    &= 
    \sum_{k=1}^n a_k'\check{\bd}(\cB_k)
    +
    \check{\bd}(\cB_j)
    -
    \sum_{i\in\cS}(r_{\mathbf{b},i}-1)_+ \\
    &= 
    \sum_{k=1}^n b_k\check{\bd}(\cB_k)
    -
    \sum_{i\in\cS}(r_{\mathbf{b},i}-1)_+,
\end{align*}
where the second inequality arises from \eqref{aggregation_induction} and \eqref{U_prime_cap_Bj}.

Since $\mathbf{b}$ is arbitrary and satisfies \eqref{b_vec}, \eqref{coherent_aggregation} therefore must hold for every vector $\mathbf{a}\in\mathbb{Z}_+^n$ whose elements sum to $L$, completing the induction proof.
\end{proof}

\subsection{Proof of \thmref{TU_equivalence}}
\begin{proof}[\unskip\nopunct]
Fix any $\cS\subseteq\cI$. It holds by \lemref{d_hat_improve} that $\tilde{\bd}(\cS)\geq\check{\bd}(\cS)$. It therefore remains to prove the reverse inequality to show that $\tilde{\bd}(\cS) = \check{\bd}(\cS)$.

Since $\cR_k = (\cR_k \cap \cS) \cup (\cR_k \backslash \cS)$ where $ (\cR_k \cap \cS)$ and $(\cR_k \backslash \cS)$ are disjoint, Lemma 3 of \citet{Goeman2021} gives
\begin{align}\label{qk_inequality}
    \check{\bd}(\cR_k\cap\cS) \geq
    \check{\bd}(\cR_k)-|\cR_k\backslash\cS| \geq \bd(\cR_k)-|\cR_k\backslash\cS| \quad \text{for $k = 1,\dots,n$.}
\end{align}
For $k=1,\dots,n$, let
\begin{equation*}
    \cB_k=\cR_k\cap\cS,
     \quad \text{and} \quad
    a_k=y_k^{\star},  
\end{equation*}
where $y_k^{\star}$ is as defined in \lemref{for_programming_representation}. Moreover, let
\begin{equation}\label{r_i}
     r_i = \sum_{k:i\in\cB_k} y_k^{\star} = \sum_{k:i\in\cR_k} y_k^{\star} \quad \text{for each $i \in \cS$.}
\end{equation}
Then
\begin{align*}
    \check{\bd}(\cS) 
    &\geq \check{\bd} \left( \bigcup_{k: a_k > 0} \cB_k \right) \\
    &\geq
    \sum_{k=1}^n
        y_k^{\star}\check{\bd}(\cB_k)
    -
    \sum_{i\in\cS}(r_i-1)_+\\
    &\geq
    \sum_{k=1}^n y_k^{\star} \left( \bd(\cR_k)-|\cR_k\backslash\cS| \right)
    -
    \sum_{i\in\cS}(r_i-1)_+\\
    &=
    \tilde{\bd}(\cS)
\end{align*}
where
\begin{itemize}
    \item the first inequality stems from coherence of $\check{\bd}$, since $\bigcup_{k: a_k > 0} \cB_k \subseteq \cS$;
    \item the second inequality is a consequence of \lemref{coherent_aggregation};
    \item the third inequality stems from \eqref{qk_inequality}; and
    \item the final equality follows from \lemref{for_programming_representation} and \eqref{r_i}.
\end{itemize} 
\end{proof}

\subsection{Proof of \propref{forest_TU}}
\begin{proof}[\unskip\nopunct]
    By the Ghouila--Houri characterization \citep[Theorem 19.3$(iv)$]{Schrijver1986}, a matrix $\mathbf{M} \in \{ 0,1 \}^{n \times |\cI|}$ is totally unimodular if, for every collection of rows $\cW\subseteq\{1,\dots,n\}$, there exist signs $\varepsilon_k\in\{-1,1\}$ such that
\begin{equation*}
    \sum_{k\in\cW}\varepsilon_k \mathbf{M}_{ki}
    \in
    \{-1,0,1\}
    \quad
    \text{for every $i\in\cI$.}
\end{equation*}

Fix any $\cW\subseteq\{1,\dots,n\}$ and suppose that the support family is forest-structured. For each row $k\in\cW$, let $h_k$ denote the number of sets in the selected collection $\{ \cR_j : j \in \cW \}$ that strictly contain $\cR_k$, i.e.,
\begin{equation*}
    h_k
    =
    \left|
        \left\{
            j\in\cW:
            \cR_k \subsetneq \cR_j
        \right\}
    \right|,
\end{equation*}
and set
\begin{equation*}
    \varepsilon_k=(-1)^{h_k}.
\end{equation*}
Fix any $i \in \cI$. Since the support family is forest-structured, any two support sets containing $i$ must be nested. Thus, the sets in $\{\cR_j:j\in\cW\}$ that contain $i$ can be written as
\begin{equation*}
    \cR_{k_1}
    \supsetneq
    \cR_{k_2}
    \supsetneq
    \cdots
    \supsetneq
    \cR_{k_{\ell_i}}
    \quad \text{where} \quad
    \ell_{i} = \sum_{j \in \cW} \mathbbm{1}\{ i \in \cR_j \}.
\end{equation*}
For each $r\in\{1,\dots,\ell_i\}$, the set $\cR_{k_r}$ has exactly $r-1$ strict supersets that are indexed by the rows in $\cW$. Hence,
\begin{equation*}
    h_{k_r}=r-1
    \quad\text{and}\quad
    \varepsilon_{k_r}=(-1)^{r-1}.
\end{equation*}
Since $\mathbf{M}_{ki}=1$ exactly when $i\in\cR_k$, it follows that
\begin{align*}
    \sum_{k\in\cW}\varepsilon_k\mathbf{M}_{ki}
    =
    \sum_{r=1}^{\ell_i}\varepsilon_{k_r} 
    =
    \sum_{r=1}^{\ell_i}(-1)^{r-1}
    \in
    \{0,1\}.
\end{align*}
Since $\cW$ and $i$ are arbitrary, the Ghouila--Houri characterization implies that $\mathbf{M}$ is totally unimodular.
\end{proof}

\section{Results and proofs relating to \thmref{iff_perfect_coherence}}\label{app:closure}
\subsection{Technical results for proving \thmref{iff_perfect_coherence}}

\begin{lemma}\label{lem:downward}
    Suppose $\bd(\cV) = 0$ for some $\cV \subseteq \cI$. Then
    \begin{equation*}
        \bd(\cU) = 0 \quad \text{for all} \quad \cU \subseteq \cV
    \end{equation*}
    if at least one of the following is true:
    \begin{enumerate}[(i)]
        \item $\bd$ is perfectly coherent. 
        \item $\bd$ is a closed testing procedure.
    \end{enumerate}
\end{lemma}
\begin{proof}[Proof of \lemref{downward}$(i)$]
    Since $\bd$ is perfectly coherent, we have that
\begin{equation}\label{minJcapV0}
     \min_{\cJ \in \cA} |\cJ \cap \cV| = \tilde{\bd}(\cV) =  \bd(\cV) = 0.
\end{equation}
Fix any $\cU \subseteq \cV$. Since
\begin{equation*}
    \tilde{\bd}(\cU)  = \min_{\cJ \in \cA} |\cJ \cap \cU| \leq \min_{\cJ \in \cA} |\cJ \cap \cV| = 0
\end{equation*}
by \eqref{minJcapV0} and $\tilde{\bd}(\cU) \geq 0$, it must be that $\tilde{\bd}(\cU) = 0$. Hence, $\bd(\cU) = \tilde{\bd}(\cU) = 0$ because $\bd$ is perfectly coherent.
\end{proof}

\begin{proof}[Proof of \lemref{downward}$(ii)$]
    Since $\bd$ is a closed testing procedure, there exists a suite of local tests
    $\varphi = (\phi_{\cW})_{\cW \subseteq \cI}$ such that
    \begin{equation*}
         \bd(\cS)
         =
         \min_{\cW \subseteq \cS}
         \{|\cS \backslash \cW| : \phi^{\cI}_{\cW} = 0 \}
         \quad \text{for all $\cS \subseteq \cI$.}
    \end{equation*}
    Since $\bd(\cV) = |\cV \backslash \cV| = 0$, we must have that $\phi^{\cI}_{\cV} =0$.
    
    Fix any $\cU \subseteq \cV$. Since $\phi^{\cI}_{\cV} = 0$ and
    \begin{equation*}
        \phi^{\cI}_{\cU}
        =
        \min\{\phi_{\cW} : \cU \subseteq \cW \subseteq \cI \} \leq  \min\{\phi_{\cW} : \cV \subseteq \cW \subseteq \cI\} = \phi^{\cI}_{\cV},
    \end{equation*}
    it must hold that $\phi^{\cI}_{\cU} = 0$. Hence, 
    \begin{equation*}
        \bd(\cU)
        =
        \min_{\cW \subseteq \cU}
        \{|\cU \backslash \cW| : \phi^{\cI}_{\cW} = 0 \}
        = |\cU \backslash \cU| =
        0.
    \end{equation*}
\end{proof}

\begin{lemma}\label{lem:cond}
    It holds that
    \begin{equation*}
        \bd(\cS) = \min_{\cU \subseteq \cS}
\left\{ |\cS \backslash \cU| : \bd(\cU)=0
\right\} \quad \text{for all} \quad \cS \subseteq \cI
    \end{equation*}
        if at least one of the following is true:
    \begin{enumerate}[(i)]
        \item $\bd$ is perfectly coherent. 
        \item $\bd$ is a closed testing procedure.
    \end{enumerate}
\end{lemma}
\begin{proof}[Proof of \lemref{cond}$(i)$]
    Fix any $\cS \subseteq \cI$ and $\cU \subseteq \cS$ such that $\bd(\cU)=0$. Since $\bd$ is perfectly coherent,
\begin{equation*}
 \min_{\cJ \in\cA}|\cJ\cap \cU| = \tilde{\bd}(\cU) = \bd(\cU) = 0.
\end{equation*}
Hence, there exists $\cJ_{\cU}\in\cA$ such that
\begin{equation*}
\cJ_{\cU} \cap \cU=\varnothing.
\end{equation*}
It therefore follows that 
\begin{equation*}
\bd(\cS) = \tilde{\bd}(\cS)
=
\min_{\cJ \in\cA}|\cJ \cap \cS|
\leq
|\cJ_{\cU}\cap \cS|
\leq
|\cS \backslash \cU|.
\end{equation*}
Thus,
\begin{equation*}
\bd(\cS) = \min_{\cU \subseteq \cS}
\left\{ \bd(\cS) : \bd(\cU)=0
\right\} \leq \min_{\cU \subseteq \cS}
\left\{ |\cS \backslash \cU| : \bd(\cU)=0
\right\}.
\end{equation*}
Now it remains to prove the converse inequality.

Since $\bd$ is perfectly coherent, there exists a set $\cJ^\star\in\cA$ such that
\begin{equation*}
|\cJ^\star\cap \cS|
= \tilde{\bd}(\cS) = 
\bd(\cS).
\end{equation*}
Define the set
\begin{equation*}
\cU^\star
=
\cS \backslash  \cJ^\star
\end{equation*}
so that $\cU^\star\subseteq \cS$ and $\cJ^\star\cap \cU^\star=\varnothing$. It follows that
\begin{equation*}
\tilde{\bd}(\cU^\star)
=
\min_{\cJ \in\cA}|\cJ\cap \cU^\star|
\leq
|\cJ^\star\cap \cU^\star|
=
0.
\end{equation*}
Thus, $\bd(\cU^\star) = \tilde{\bd}(\cU^\star)=0$ by perfect coherence, and so,
\begin{equation*}
\min_{\cU \subseteq \cS}
\left\{ |\cS \backslash \cU| : \bd(\cU)=0
\right\}
\leq
|\cS\backslash  \cU^\star|
=
|\cJ^\star \cap \cS|
=
\bd(\cS).
\end{equation*}
\end{proof}

\begin{proof}[Proof of \lemref{cond}$(ii)$]
    Since $\bd$ is a closed testing procedure, there exists a suite of local tests $\varphi = (\phi_{\cU})_{\cU \subseteq \cI}$ such that
    \begin{equation*}
        \bd(\cS) = \bd_{\varphi}(\cS) = \min_{\cU \subseteq \cS} \{ |\cS \backslash \cU|: \phi^{\cI}_{\cU} = 0 \} \quad \text{for all $\cS \subseteq \cI$.}
    \end{equation*}
    Fix any $\cS \subseteq \cI$. If $\bd_{\varphi}(\cS) = 0$, then $\phi^{\cI}_{\cS} = 0$ must hold since $|\cS \backslash \cS| = 0$. If $\phi^{\cI}_{\cS} = 0$, then $\bd_{\varphi}(\cS)  = |\cS \backslash \cS| = 0$. Hence, $\bd_{\varphi}(\cV) = 0$ if and only if $\phi^{\cI}_{\cV} = 0$, which means that
    \begin{equation*}
        \bd_{\varphi}(\cS) = \min_{\cU \subseteq \cS}\{ |\cS \backslash \cU|: \phi^{\cI}_{\cU} = 0 \} = \min_{\cU \subseteq \cS}\{ |\cS \backslash \cU|: \bd_{\varphi}(\cU) = 0 \}.
    \end{equation*}
\end{proof}

\subsection{Proof of \thmref{iff_perfect_coherence}}
We prove \thmref{iff_perfect_coherence} in two parts. First, we show that if $\bd$ is perfectly coherent, then $\bd$ is a closed testing procedure. Then we prove the converse. 

\begin{proof}[Proof that if $\bd$ is perfectly coherent, then $\bd$ is a closed testing procedure]
Suppose $\bd$ is perfectly coherent. For each $\cS \subseteq \cI$, let
    \begin{equation*}
        \phi_{\cS} = \mathbbm{1}\{  \bd(\cS) > 0 \}
    \end{equation*}
    be the local test for the intersection hypothesis $\cH_{\cS}$. Then
    \begin{align*}
       \sup_{\P \in \cH_{\cS}} \P \left( \phi_{\cS} = 1 \right) 
        = \sup_{\P \in \cH_{\cS}} \P \left( \bd(\cS) > 0  \right) 
        = \sup_{\P \in \cH_{\cS}} \P \left( \bd(\cS) > |\cS \cap \cI_1|  \right) \leq \alpha,
    \end{align*}
    noting that $|\cS \cap \cI_1| = 0$. Hence, $\phi_{\cS} = \mathbbm{1}\{  \bd(\cS) > 0 \}$ is a valid local test for $\cH_{\cS}$. For the corresponding suite $\varphi = (\phi_{\cS})_{\cS \subseteq \cI}$, the effective local test for any subset $\cU \subseteq \cI$ is
    \begin{align*}
        \phi^{\cI}_{\cU}
        & = \min\{ \mathbbm{1}\{ \bd(\cV) > 0 \} : \cU \subseteq \cV \subseteq \cI \}.
    \end{align*}
    Fix any $\cU \subseteq \cI$. If $\bd(\cU) = 0$, then $\phi^{\cI}_{\cU} = 0$. If $\phi^{\cI}_{\cU} = 0$, then there exists a $\cV^\star$ such that $\bd(\cV^\star) = 0$ where $\cU \subseteq \cV^\star$, which implies that $\bd(\cU) = 0$ by \lemref{downward}$(i)$. Hence, $\bd(\cU) = 0$ if and only if $\phi^{\cI}_{\cU} = 0$, which means that
    \begin{align*}
        \bd_{\varphi}(\cS) = \min_{\cU \subseteq \cS}\{ |\cS \backslash \cU|: \phi^{\cI}_{\cU} = 0 \}  = \min_{\cU \subseteq \cS}\{ |\cS \backslash \cU|: \bd(\cU) = 0 \} = \bd(\cS),
    \end{align*}
    where the final equality is a result of \lemref{cond}$(i)$. Thus, $\bd$ is a closed testing procedure.
\end{proof}

\begin{proof}[Proof that if $\bd$ is a closed testing procedure, then $\bd$ is perfectly coherent]
Fix any $\cS \subseteq \cI$. Since $\bd$ is a closed testing procedure, there exists a set $\cU^\star\subseteq \cS$ by \lemref{cond}$(ii)$ such that
\begin{equation*}
    \bd(\cU^\star)=0
    \quad \text{and} \quad
    \bd(\cS)=|\cS\backslash \cU^\star|.
\end{equation*}
Define the set
\begin{equation*}
    \cJ^{\star} = \cI \backslash \cU^\star
\end{equation*}
so that
\begin{equation}\label{equal_maximum}
    \bd(\cS) 
    =
    |\cS\backslash \cU^\star|
    = |\cJ^{\star}\cap \cS|.
\end{equation}
To complete this proof, it suffices to show that
\begin{equation}\label{JstarcapSequalminJcapS}
    |\cJ^{\star}\cap \cS| = \min_{\cJ \in \cA}|\cJ \cap \cS|,
\end{equation}
since this, together with \eqref{equal_maximum}, would imply that $\bd(\cS) = \tilde{\bd}(\cS)$, i.e., $\bd$ is perfectly coherent.

Fix any $\cV\subseteq \cI$. Since $\cV \cap \cU^\star \subseteq \cU^\star$ and $\bd(\cU^\star)=0$, we have by \lemref{downward}$(ii)$ that $\bd(\cV \cap \cU^\star)  = 0$. Therefore, 
\begin{equation*}
      \bd(\cV)
    =
    \min_{\cU\subseteq \cV}\{|\cV\backslash  \cU|:\bd(\cU)=0\}
    \leq
    |\cV\backslash  (\cV\cap \cU^\star)|
    =
    |\cV\backslash  \cU^\star|.
\end{equation*}

Since $|\cV \backslash \cU^\star| = |\cJ^\star \cap \cV|$, it follows from the above display that
\begin{equation*}
    \bd(\cV)\leq |\cJ^{\star}\cap \cV|
    \qquad\text{for all } \cV \subseteq \cI,
\end{equation*}
which means that $\cJ^{\star} \in \cA$.

By \lemref{d_hat_improve} and \eqref{equal_maximum}, we have that 
\begin{equation}\label{JS_upper}
    \min_{\cJ \in \cA} |\cJ \cap \cS| \geq \bd(\cS) = |\cJ^{\star} \cap \cS|.
\end{equation}
Since it was established that $\cJ^{\star} \in \cA$, we must also have that
\begin{equation}\label{JS_lower}
   \min_{\cJ \in \cA} |\cJ \cap \cS| \leq |\cJ^{\star} \cap \cS|.
\end{equation}
Combining \eqref{JS_upper} and \eqref{JS_lower} therefore gives \eqref{JstarcapSequalminJcapS}.

\end{proof}

\end{document}